\documentclass[11pt, twoside]{article}
\usepackage{amsfonts,amssymb,amsmath,amsthm}
\usepackage{graphicx}
\usepackage[all]{xy}
\usepackage{multirow} 
\usepackage{psfrag,xmpmulti,amscd,color,pstricks, import}
\usepackage{tikz-cd}

 \divide\oddsidemargin by 2
\newtheorem{theorem}{Theorem}[section]
\newtheorem{lemma}[theorem]{Lemma}
\newtheorem{corollary}[theorem]{Corollary}
\newtheorem{proposition}[theorem]{Proposition}

\theoremstyle{definition}
\newtheorem{definition}[theorem]{Definition}

\newtheorem{problem}[theorem]{Problem}
\newtheorem{remark}[theorem]{Remark}

\newtheorem*{thmA}{Theorem A}
\newtheorem*{thmB}{Theorem B}
\newtheorem*{thmC}{Theorem C}

\numberwithin{equation}{section}

\definecolor{OrangeRed}{cmyk}{0,0.6,1,0}            
\definecolor{DarkBlue}{cmyk}{1,1,0,0.20}
\definecolor{DarkGreen}{cmyk}{1,0,0.6,0.2}
\definecolor{myblue}{rgb}{0.66,0.78,1.00}
\definecolor{Violet}{cmyk}{0.79,0.88,0,0}
\definecolor{Lavender}{cmyk}{0,0.48,0,0}

\renewcommand{\epsilon}{\varepsilon}
\renewcommand{\phi}{\varphi}

\title{Orbits Inside Basins of Attraction of Skew Products }

\vspace{5cm}
\author{   
	John Erik Forn\ae ss,\\
	\small Department of Mathematical Sciences, Norwegian University of Science and Technology,\\
	\small Trondheim, 7034, Norway\\
	\small fornaess@gmail.com
	\and	Mi Hu \\
	\small Department of Mathematics, University of Oslo, Oslo, 0371, Norway\\
	\small humihqu@gmail.com\\
	\\
	\large		Im memory of Joe Kohn
}

\begin{document}

	\maketitle
	
	\begin{abstract}

		A basic problem in complex dynamics is to understand orbits
		of holomorphic maps. One problem is to understand the
		collection of points $S$ in an attracting basin whose forward orbits
		land exactly on the attracting fixed point. 
		In the paper \cite{RefH1},
		the second author showed that for holomorphic polynomials in $\mathbb C$, there is a constant $C$ so that all Kobayashi discs of radius $C$ must intersect this set $S$.
		In the paper \cite{RefH3}, the second author showed there are holomorphic skew products in $\mathbb C^2$ where this result fails. The main result of this paper is to show that for a large class of polynomial skew products, 
		this result nevertheless holds.
		
	\end{abstract}
	
		{\bf keywords: } {Basin of attraction, \and Kobayashi metric, \and Skew Products}        
	
	{\bf MRSubClass:} {32H50, 32F45}  
	
	\section{Introduction}
	
	Let $F: M\rightarrow M$ be a holomorphic selfmap of a complex manifold $M$. We assume that $F$ has an attracting fixed point
	$p.$ Let $\Omega$ denote the attracting basin consisting
	of those $q\in M$ for which $\{F^n(q)\}_{n=1}^\infty
	\rightarrow p$. We let $S$ denote the subset of those $q\in \Omega$ for which there exists an integer $n\geq 1$ so that
	$F^n(q)=p$. We know that if an orbit $\{z_n\}$ starting from $z_0$ eventually reaches the point $p$ for some integer $n$, then the behavior of the orbit $\{z_n\}$ is clear. 
	However, in most cases, the orbit $\{z_n\}$ of $z_0$ approaches $p$ indefinitely without ever reaching it. In such situations, when we want to know the precise behavior of the orbit of $z_0$, it becomes complicated. Then we ask whether we can investigate the backward orbit of the fixed point $p$ instead. Hence, no matter whether $M=\mathbb{C}, \hat{\mathbb{C}}=\mathbb{C}\cup\{\infty\}$, or $\mathbb{C}^2$, how to handle these situations becomes an interesting question. In this paper, we investigate the following problem:
	
	\begin{problem}\label{prob}
	Does there exist a constant $C$ such that for every point $z_0 \in \Omega$ there is a point $q \in S$ such that the Kobayashi distance between $z_0$ and $q$ is less than $C$? 
	
	\end{problem}
	
The reason why we use the Kobayashi metric is that the Kobayashi metric is an important tool in complex dynamics, see examples in \cite{RefA, RefBracci, RefBGN, RefBRS}. It has many good properties: it is distance-decreasing and invariant under holomorphic conjugation while the Euclidean distance is not. 

Furthermore, there are two important subsets of $\hat{\mathbb C}$, which are called the Fatou set and the Julia set, see references \cite{RefB, RefCG, RefFatou, RefM}. There have been many studies on probability measures that can describe the dynamics on the Julia set
For example, if $z$ is any non-exceptional point,  the inverse orbits $\{f^{-n}(z)\}$ equidistribute toward the Green measure $\mu$, which lives on the Julia set. This was already proved by Brolin \cite{RefBrolin} in 1965, and many improvements and generalizations have been made \cite{RefDS, RefDO}. However, this equidistribution toward $\mu$ is in the weak sense, and hence it is with respect to the Euclidean metric. Therefore, it is a reasonable question to ask how dense $\{f^{-n}(z)\}$ is in the Fatou set near the boundary of the Fatou set if we use finer metrics, for instance, the Kobayashi metric. See the definition of the Kobayashi metric in Section 2.

	In the paper \cite{RefH1}, Hu showed that the answer is affirmative for
	polynomials in $\mathbb C$. 
	She studied the dynamics of holomorphic polynomials on attracting basins and obtained Theorem A:
	
	\begin{thmA}\label{thmA} (Hu, \cite{RefH1})
		Suppose $f(z)$ is a polynomial of degree $N\geq 2$ on $\mathbb{C}$, $p$ is an attracting fixed point of $f(z),$ $\Omega_1$ is the immediate basin of attraction of $p$, $\{f^{-1}(p)\}\cap \Omega_1\neq\{p\}$,
		$\mathcal{A}(p)$ is the basin of attraction of $p$, $\Omega_i (i=1, 2, \cdots)$ are the connected components of $\mathcal{A}(p)$. Then there is a constant $\tilde{C}$ so that for every point $z_0$ inside any $\Omega_i$, there exists a point $q\in \cup_k f^{-k}(p)$ inside $\Omega_i$ such that $d_{\Omega_i}(z_0, q)\leq \tilde{C}$, where $d_{\Omega_i}$ is the Kobayashi distance on $\Omega_i.$  
	\end{thmA}

	Theorem A shows that in an attracting basin of a complex polynomial, the backward orbit of the attracting fixed point either is the point itself or accumulates at the boundary of all the components of the basin in such a way that all points of the basin lie within a uniformly bounded distance of the backward orbit, measured with respect to the Kobayashi metric. This is an interesting and innovative problem and result. There are no other publications by other researchers.
	In the paper \cite{RefFH}, the authors showed that the same result also holds for all rational functions on $\hat{\mathbb C}$.
	However, they found in the same paper that there are entire functions on $\mathbb C$ for which the answer to the question is negative, and in the paper \cite{RefH2}, the analogous result does not hold for parabolic basins.
	
	 In the paper \cite{RefH3}, Hu investigated the situation in dimension two
	 , i.e., $M\subset\mathbb C^2$.
	Compared with one-dimensional cases, there are very interesting results about dynamics inside attracting basins. 
	This involves delicate estimates of the Kobayashi metric on domains in $\mathbb C^2$.

	Firstly, Hu obtained similar results on $\mathbb C^2$ as Theorem A on $\mathbb C$. 
	
	\begin{thmB}[Hu, \cite{RefH3}]
		Suppose $F(z, w)=(P(z), Q(w)),$ where $P(z), Q(w)$ be two polynomials of degree $m_1, m_2\geq2$ on $\mathbb{C},$
		$P(0)=0, Q(0)=0, 0<|P'(0)|, |Q'(0)|<1.$  
		Let $\Omega$ be an immediate attracting basin of $F(z, w)$. 
		Then there is a constant $C$ such that for every point $(z_0, w_0)\in \Omega$, there exists a point $(\tilde{z}, \tilde{w})\in \cup_k F^{-k}(0, 0), k\geq0$ so that $d_\Omega\big((z_0, w_0), (\tilde{z}, \tilde{w})\big)\leq C,  d_\Omega$ is the Kobayashi distance on $\Omega$. 
	\end{thmB}

	However, Hu also showed Theorem B is not valid for some other cases of $F(z, w)=(P(z), Q(w))$. For example, $P(z)=z^2, Q(w)=w^2.$ For more cases, we refer to the paper \cite{RefH3}. 
	
	In addition, she not only considered $F(z, w)=(P(z), Q(w))$ but also studied 
	 polynomial skew products, i.e. 
	 maps of the form $F(z,w)=(P(z), Q(z,w))$ which have been investigated by many authors. Their great advantage is that one can obtain a lot of information from using the theory of iteration
	for holomorphic polynomials $P(z)$ in one variable.
In the paper \cite{RefH3}, she showed that there are
	skew products for which the answer to the problem is negative.
	In fact, there is the following counterexample:
	
		Suppose $F$ is a polynomial skew product, $F(z, w)=(P(z), Q(z, w)).$ 
	Theorem B fails in the following cases:
	
	1. $P(z)=z^2, Q(z, w)=w^2+az, a\in\mathbb{C};$ 
	
	2. $P(z)=az+z^2, Q(z, w)=w^2+cw+bz, 0<|a|, |b|, |c|<<1$ and $ |a|>>|c|, |a|>>|b|, |c|>>|ab|$.
	
	The proof used very involved estimates of the Kobayashi metric
	on the basin of attraction. Observe that such basins are usually fractal and the Kobayashi metric is very complicated.
	
	The goal of this paper is to investigate a fairly general criterion which will
	suffice to give an affirmative answer to Problem \ref{prob} for $M=\mathbb{C}^2$. This will apply to a large class of maps. We also show by example how such a criterion can be verified.  We also show why the counterexample of the second case above fails to satisfy the criterion. 
	
		In this paper, we will define a class  $\mathcal A$ of skew products in Section 2. 	Our main result is the following theorem.
		\begin{thmC}
			Suppose that $F\in \mathcal A$.
			Then there exists a constant $C$ such that for every point $(z,w)\in \Omega$ there is a point $(z_0,w_0)$ in $\Omega$ and an integer $N$ such that
			the Kobayashi distance between $(z,w)$ and $(z_0,w_0)$  in $\Omega$
			is at most $C$ and $F^N(z_0,w_0)=(0,0).$
		\end{thmC}
		 In Section 3, we will prove Theorem C, i.e., Theorem \ref{Mainthm}. 
		 In Section 4, we give an example to show that the class $\mathcal{A}$ is non-empty. 
	
	\section{Preliminaries}
	
	\subsection{The Kobayashi metric}
	For the convenience of the reader, we recall the definition of the Kobayashi metric.
	
	\begin{definition}\label{def1.2.23}
		Let $M$ be a complex manifold. We choose a point $z\in M$ and a tangent vector $\xi$ to $M$ at the point $z.$ Let $\triangle$ denote the unit disk in the complex plane.
		We define the {\bf Kobayashi metric} \cite{RefK}
		$$
		F_M(z, \xi):=\inf\{\lambda>0 : \exists f: \triangle\stackrel{hol}{\longrightarrow} M, f(0)=z, \lambda f'(0)=\xi\}.
		$$

		Let $\gamma: [0, 1]\rightarrow M$ be a piecewise smooth curve.
		The {\bf Kobayashi length} of $\gamma$ is defined to be 
		$$ L_{M} (\gamma)=\int_{\gamma} F_{M}(z, \xi) \lvert dz\rvert=\int_{0}^{1}F_{M}\big(\gamma(t), \gamma'(t)\big)\lvert \gamma'(t)\rvert dt.$$

		For any two points $z_1$ and $z_2$ in $M$, the {\bf Kobayashi distance} between $z_1$ and $z_2$ is defined to be 
		$$d_{M}(z_1, z_2)=\inf\{L_{M} (\gamma): \gamma ~ \text{is a piecewise smooth curve connecting} ~z_1~ \text{and} ~z_2 \}.$$

		Note that $d_{M}(z_1, z_2)$ is defined when $z_1, z_2$ are in the same connected component of $M.$

		
	\end{definition}

		\subsection{A class $\mathcal A$ of skew products}

	We will investigate polynomial skew products of the form
	$F(z,w)=(P(z), Q(z,w))$. We will make several hypotheses. 
	Many of them are just for convenience to reduce the technicalities in the proof. It will be clear to the interested reader that they
	can be relaxed.
	
	We make a basic hypothesis on $P:$
	\begin{itemize}
		\item $P(0)=0$;
		\item $0<|P'(0)|=a<1;$
		\item The degree $d$ of $P$ is at least 2.
		\end{itemize}

	Let $U$ denote the basin of attraction of $P$ and let
	$U_0$ denote the connected component of $U$ containing the origin and let $U_j,j>0$ denote the other connected components of $U$.
	Recall from the above Theorem \ref{thmA} that there is a constant $C$ such that 
	if $z\in U_j$ then there is a  preimage $z_0=P^{-N}(0)$
	in $U_j$ such that the Kobayashi distance in $U_j$ from
	$z$ to $z_0$ is at most $C.$

	We also make basic hypotheses on $Q.$
	We can write $Q(z,w)=\sum_{j=0}^d Q_j(w)z^j.$ 
	\begin{itemize}
		\item The degree of $Q_j, j>1$ is at most $d-1$;
		\item $Q_0(w):=Q(0,w)$ is a polynomial of degree $d$;
		\item $Q_0(0)=0$;
		\item $0<|Q'_0(0)|=b<1$;
		\item $a<b$;
		\item $\frac{\partial Q}{\partial z}(0,0)=0.$
	\end{itemize}
	
	The conditions of $P,Q$ imply that $(0,0)$ is an attracting fixed point. Let $\Omega$ be the basin of attraction, let
	$\Omega^0$ denote the immediate basin of attraction and let
	$\Omega^j, j>0,$ denote the other connected components of the basin (if there are any).
	
	For any $z\in \mathbb C,$ we let $\Omega_z=\{w; (z,w)\in \Omega\}$ be the slice of $\Omega$ at $z$.

	The first item in the requirements of $Q$  implies that the basin of attraction,  $\Omega$,
	is a bounded set which will be convenient.
	
Next we make some conditions on the critical set. The most important condition is in the first item. The others can be relaxed and are there to simplify technicalities.

	
	\begin{definition}\label{def2}
		We say that $F$ satisfies condition $\mathcal{C}$ if the following holds:
		\begin{itemize}
			\item Let $z\in\partial U$ be any point. Then there is no point $(z,w)
			\in \overline{\Omega}$ for which $\frac{\partial Q}{\partial w}=0.$
			\item Let $w\notin \Omega_0$ and 
			$\frac{\partial Q}{\partial w}(0,w)= 0,$ then $(0,w)$ is in the basin of infinity of $Q_0.$
			\item Suppose that $w\in \Omega_0$ and $\frac{\partial Q}{\partial w}= 0$. Then $F^n(0,w)\neq (0, 0)$ for all positive
			integers  $n.$
			\item If $z\in U$ and $\frac{\partial P}{\partial z}=0$
			then $P^n(z)\neq 0$ for all positive integers $n$.
		\end{itemize}
		
	 We denote by $\mathcal A$ the class of maps satisfying all the above conditions, including condition $\mathcal{C}$.
		
	\end{definition}
	
	\section{The dynamics of skew products of class $\mathcal A$}
	
		In this section, we will discuss the dynamics of skew products of class $\mathcal A$ and present our main result.
	
	\begin{theorem}\label{Mainthm}
		Suppose that $F\in \mathcal A$.
		Then there exists a constant $C$ such that for every point $(z,w)\in \Omega$ there is a point $(z_0,w_0)$ in the same Fatou component  $\Omega^j$ as $(z,w)$ and an integer $N$ such that
		the Kobayashi distance between $(z,w)$ and $(z_0,w_0)$  in $\Omega^j$
	is at most $C$ and $F^N(z_0,w_0)=(0,0).$
	\end{theorem}
	
	The conclusion of Theorem \ref{Mainthm} fails for some $F$. As mentioned above there is a counterexample in the paper \cite{RefH3}, which is $F(z, w)=(P(z), Q(z, w)), P(z)=az+z^2, Q(z, w)=w^2+cw+bz, 0<|a|, |b|, |c|<<1$ and $ |a|>>|c|, |a|>>|b|, |c|>>|ab|$. This example does not belong to the class $\mathcal A$ since it fails the first item of condition $\mathcal{C}$ in Definition \ref{def2}. Furthermore, we will also give an explicit example in the next section to show that the class $\mathcal{A}$ is non-empty. 
	Therefore, we give an affirmative answer to Problem \ref{prob} for many skew products.

	\begin{lemma}
	The set $\Omega_z$ is open and uniformly bounded. Moreover,
	the set is nonempty if and only if $z\in U$.
	\end{lemma}
	
	\begin{proof}
The hypotheses on $F$ immediately imply that there is a constant $R>0$ such that if $\max\{|z|,|w|\}>R$ then
$F^n(z,w)\rightarrow \infty$ when $n\rightarrow \infty$.
Hence, the basin of attraction of $(0,0)$
is contained in $\Delta^2((0,0),R).$
Since $\Omega$ is open, it follows that each $\Omega_z$ is open
and in fact contained in $\Delta(0,R).$
If $(z,w)$ is a point where $z\notin U$, then
$P^n(z)$ cannot converge to $0.$ Hence
$\Omega_z=\emptyset.$
If $z\in U,$ then for all large $n,$ the point
$(P^n(z),0)\in \Omega.$ But then all points in
$F^{-n}(P^n(z),0)$ belong to $\Omega.$ This implies that $\Omega_z$ is non-empty.
	\end{proof}
	
	\begin{lemma}
	Let $z\in U.$ Then $\Omega_{P(z)}=\{Q(z,w); w\in \Omega_z\}$.
	Moreover, the map $w\rightarrow Q(z,w)$ from $\Omega_z$ to $\Omega_{P(z)}$ is a branched cover
	with branching number $d.$
\end{lemma}
	
	\begin{proof}
		The lemma follows since for each fixed $z$, the function $Q(z,w)$
	is a polynomial of degree $d.$ 
	\end{proof}
	
	Using the maximum principle we obtain the following Lemma:
	
	\begin{lemma}
	The connected components of each $\Omega_z$ are simply connected. 
	\end{lemma}
	
	The next Lemma follows from the first item of Definition \ref{def2}.
		
	\begin{lemma}
There is a compact subset $K$ of $U$ such that if
$z\in U$ and $z,P(z)$ are outside $K$, then there are no branch points of the map $\Omega_z$ to $\Omega_{P(z)}.$
In particular, these maps are all isometries in the Kobayashi metric
inside all connected components.
	\end{lemma}

	Our hypothesis on the eigenvalues of $F'(0,0)$ implies that there is a strong stable curve $\Sigma$ in a neighborhood of the origin tangent to the $z$ axis, and a weak stable manifold tangent to  the $w$ axis. 
	Let $(0,w_n)$ denote the preimages of $(0,0)$ with $w_0=0.$
	
	\begin{lemma}
		There exists an $\epsilon >0$ so that for every
		$w_n$ there exists a complex manifold $\Sigma_n=\{w=f_n(z),|z|<\epsilon, f_n(0)=w_n
		\}$ which are preimages of $\Sigma.$ These are pairwise disjoint.
		Moreover, there are no other preimages $(z,w)$ of $\Sigma$
		with $|z|<\epsilon.$
	\end{lemma}
	
	\begin{proof}
 We first observe that we can take $\Sigma_0$ to be a piece of $\Sigma$
 which is a graph over the $z-$ axis. If $Q(0,w_1)=0,$ then we observe that $F$ is biholomorphic in a neighborhood of $(0,w_1)$ to 
 a neighborhood of $(0,0)$ and the map is contracting in the $z$ direction. We now define $\Sigma_1$ as the preimage of $\Sigma_0$
 which we restrict to $|z|<\epsilon.$ The other $\Sigma_n$
 are obtained by induction. The last part of this Lemma is obvious. 
	\end{proof}
	
	\begin{lemma}
	There exists a constant $C_1$ so that if $|z|<\epsilon$ 
	and $w\in \Omega_z$, then there is an integer $n$ so that
	the Kobayashi distance in $\Omega_z$ between $(z,w)$ and
	$(z,f_n(z))$ is at most $C_1.$
	\end{lemma}
	
	\begin{proof}
	
	Let $K_0=\{\overline{\Delta^2((0,0), \epsilon)}\}$. Shrinking $\epsilon$ if necessary this is
		contained in the basin of attraction $\Omega.$
		Inductively define $K_{n+1}=F^{-1}(K_n)\cap \{|z|\leq \epsilon$\}.
		This is an increasing sequence of compact subsets of $\Omega$.
		For all large $n$, the sets $K_n$ contain all points $(z,w)$ in $\Omega$ 
		where $|z|\leq \epsilon$ and $\frac{\partial Q}{\partial w}=0.$
		We fix a large $n$ and consider the
		map $F(z, w)$ from $\Omega_z\setminus K_n$ to $\Omega_{P(z)}\setminus K_{n-1}$ when $|z|\leq \epsilon$. This is an unbranched covering. Hence, it is a local isometry in the Kobayashi metric. If we restrict
		to those points whose image also are contained in the complement of $K_n$ and still using the Kobayashi metric on
		the complement of $K_n$ in the respective slices, the map $F$
		becomes distance increasing.
		 It follows now that if the slices $\Omega_z$ contain arbitrarily large Kobayashi discs, then this happens as well for $z$ arbitrarily close to $0.$
		 For a fixed $C$, there will exist a $z$ for which $\Omega_z$ contains a Kobayashi disc of radius at least $C$. This disc will be contained in some $K_m.$ Pushing forward we find the same to be true for points arbitrarily close to $0$. Hence, it follows that
		 the same holds for the slice $\Omega_0.$ 
		 But  this contradicts Theorem A above applied to $\Omega_0.$
		 
		 \end{proof}
		 
		We next prove a global version of the Lemma.
		\begin{lemma}
			There exists a constant $C_2$ such that the following holds:
			Let  $z\in U$ 
			and $w\in \Omega_z$. Then there exists a point $(z,\eta)\in \Omega_z$ so that the Kobayashi distance between $(z,w)$
			and $(z,\eta)$  in $\Omega_z$ is at most $C_2$. Moreover, there exists an $N$ such that $|P^N(z)|<\epsilon$
			and such that $F^N(z,\eta)$ belongs to one of the
			graphs $\Sigma_n.$
		\end{lemma}
		
		\begin{proof}
			Let $z\in U$. Let $N$ be the smallest integer so that
			$|P^N(z)|<\epsilon.$ Consider the sequence of
			slice maps $f_1:\Omega_z\rightarrow \Omega_{P(z)}$,
			$f_2:\Omega_{P(z)}\rightarrow \Omega_{P^2(z)}$,
			..., $f_N:\Omega_{P^{N-1}(z)}\rightarrow \Omega_{P^N(z)}$.
			These are branched covers of degree $d$. But by the first
			item in condition $\mathcal C$ at most $\tilde{M}$ of these actually have branch points, where $\tilde{M}$ is a uniform bound.
			We consider the points on the
			slice $P^N(z)$ which belong to the $\{\Sigma_n\}$.
			By the previous Lemma, no Kobayashi disc of radius $C$
			in the slice can avoid all these points. Then the same holds 
			inductively backwards for all the slices $P^n(z)$ for $0\leq n<N$ except that the constant $C$ must be replaced by a larger constant whenever the slice map has a branch point.
			Moreover, the increase in the constant is given by the 
			branching order which is uniform. It follows that there is a constant $C_2$ as claimed.
			\end{proof}
			
		We are now ready to prove our main theorem, Theorem 3.1:

		\begin{proof}
		First, let $\Lambda_0=\{w=f_0(z), |z|<\epsilon\}.$
		Inductively, let $\Lambda_{n+1}=F^{-1}(\Lambda_n).$
		Let $\Lambda=\cup_n\Lambda_n.$
		Let $D_0=\{(z,w), |z|,|w|<\epsilon\}$ and let inductively
		$D_{n+1}=F^{-1}(D_n).$ Then the sequence $D_n$ exhaust $\Omega$ and
		all $\Lambda_n$ are closed subvarieties of $D_n.$

			Note that all points in $D_n$ have a bounded Kobayashi distance (depending on $n$) from the origin because $D_n$ is relatively compact in $\Omega.$
			
			Let $(z,w)\in \Omega$ be a point in $D_{n+1}\setminus D_n$
			for some large $n.$ By Lemma 3.8 there is a point $(z,\eta)$ in $\Lambda$ belonging to $\Omega_z$ with Kobayashi distance in $\Omega_z$ at most $C_2$. Then $F^{n+1}(z,w)=(a,b)$
			with $|a|,|b|<\epsilon.$ This implies that
			$F^{n+1}(z,\eta)$ is on some graph $w=f_m(z)$ and the Kobayashi distance in the slice $\Omega_a$ is at most $C.$
			This implies that $m\leq m_0$ for a uniform $m_0.$
			This implies that $(z,\eta)\in U_{n+M}$ for a uniform $M.$
			Similarly we can suppose that $(z, \eta)\notin U_{n-M}$.
			Consider the orbit $(z_m,\eta_m)$ where
			$(z_0, \eta_0)=(z,\eta)$ and $z_{m+1}, \eta_{m+1}=F(z_m,\eta_m)$
			Let $m'$ be the first number where $z_{m'}, \eta_{m'}\in \Lambda _0$.
			Let $\Sigma_0=\Lambda_0$ and consider the closed subvariety of $U_1$ given by $F^{-1}(\Lambda_0).$
			This is a multisheeted branched cover over $P^{-1}(|z|<\epsilon).$
			Let $\Sigma_1$ be the irreducible branched cover over 
			$P^{-1}(|z|<\epsilon)\setminus \{|z|<\epsilon\}$ which contains
			$(z_{m'-1},\eta_{m'-1})$. We continue inductively to 
			obtain finally $\Sigma_{m'}$ which is branched  over
			$U_{m'+1}\setminus U_{m'}$ and contains both $(z,\eta)$ and a preimage $p$ of the origin.
			Notice also that the map $F$ is locally biholomorphic
			on $U_{n+1}\setminus U_n$ so $F$ is a Kobayashi isometry
			on $\Lambda$ as well.
			Hence the Kobayashi distance between $(z,\eta)$ and
			$p$ on $\Sigma_{m'}$ is uniformly bounded, but then it is also uniformly bounded in the Kobayashi metric on $\Omega.$
			\end{proof}

			\section{An example where the Theorem applies}
			
			We just show that the map $F(z, w)=(P(z), Q(z,w))=(z^2+\frac{z}{4}, 
			w^2+\frac{w}{2}+Lz^2)$ is in $\mathcal A$ for suitable $L$ to be determined.
			We see that $\{|z|<\frac{3}{4}\}\subset U\subset \frac{5}{4}.$ 
			Hence, $\partial U\subset \{3/4\leq |z|\leq 5/4\}.$
			
			The critical point $\frac{\partial Q}{\partial w}=0$
			is given by $w=-\frac{1}{4}$.
			
			\begin{lemma}
				The basin of attraction of $(0,0)$ is contained in
				the set $|z|< \frac{5}{4}, |w|< B(L)$, where $B(L)-\frac{25L}{16 B(L)}\geq \frac{3}{2}$.
				\end{lemma}
				
				\begin{proof}
					Clearly $|P(z)|>|z|$ when $|z|>\frac{5}{4}.$
					When $|w|>B(L)$, we get that
				$|Q(z, w)|>|w| B(L)-|w|/2-\frac{25 L}{16},$
				so $|Q(w)>|w|(B(L)-1/2-\frac{25L}{16B(L)}).$
					So if condition (*) is satisfied,
					$$
					(*) B(L)-1/2-\frac{25L}{16 B(L)}\geq 1. 
					$$ 
					Then conclusion holds.
				\end{proof}
				
				\begin{lemma}
				If $z\in \partial U,$ then there is no $w$ for which
				$\frac{\partial  Q}{\partial w}=0 $ and $(z,w)\in \overline{\Omega}.$
				\end{lemma}

			\begin{proof}
			If $z\in \partial U$ and $w=1/2,$ then
			$|Q(z,w)| \geq \frac{9}{16}L-1/16$.
			We need 
			$$
			(**)\frac{9}{16}L-1/16>B(L).
			$$
			\end{proof}
			
			We next look for a number $L$ so that both $(*)$ and $(**)$ are satisfied.
			We see that $L=10, B(L)=5$ works.

			\end{document}